\documentclass[paper=letter, fontsize=10pt,leqno]{scrartcl}
\usepackage[T1]{fontenc}
\usepackage[english]{babel}									
\usepackage[protrusion=true,expansion=true]{microtype}		
\usepackage{amsmath}										
\usepackage{amsfonts}
\usepackage{amsthm}
\usepackage{latexsym}
\usepackage{placeins}
\usepackage{soul}

\usepackage[pdftex]{graphicx}														
\usepackage{url}
\usepackage{amssymb}
\usepackage{bbm}
\usepackage{MnSymbol}

\usepackage[pdftex]{color}
\usepackage{eucal}
\usepackage{amscd}

\usepackage{sectsty}												
\allsectionsfont{\centering \normalfont\scshape}	

\usepackage{mathtools}

\usepackage[]{forest}

\usepackage{fancyhdr}
\newtheorem{theorem}{Theorem}

\newtheorem{proposition}{Proposition}
\newtheorem{lemma}[proposition]{Lemma}
\newtheorem{corollary}{Corollary}
\newtheorem{definition}{Definition}

\newtheorem{remark}{Remark}

\numberwithin{equation}{section}		 
\numberwithin{proposition}{section}			 
\numberwithin{table}{section}				 
\numberwithin{definition}{section}
\numberwithin{theorem}{section}
\numberwithin{corollary}{section}
\numberwithin{exercise}{section}

\title{
		\vspace{-1in} 	
		\usefont{OT1}{bch}{b}{n}
		\normalfont \normalsize \textsc{} \\ [25pt]
		\huge  No-Slip Billiards in Dimension Two
}

\author{\normalfont \large 
   C. Cox\footnote{Department of Mathematics, Washington University, Campus Box 1146, St. Louis, MO 63130}, 
\ R. Feres\footnotemark[1] 
}

\begin{document}

\maketitle

\begin{center}
{\em Dedicated to the memory of Kolya Chernov}
\end{center}

\vspace{.1 in}

\begin{abstract}
\begin{center} Abstract \end{center}
{\small We investigate the dynamics of no-slip billiards, a model in which small rotating disks may exchange linear and angular momentum at collisions with the boundary. We give new results on periodicity and boundedness of orbits which suggest that a class of billiards (including all polygons) is not ergodic. Computer generated phase portraits demonstrate non-ergodic features, suggesting chaotic no-slip billiards cannot readily be constructed using the common techniques for generating chaos in standard billiards. 
}
\end{abstract}

\section{Introduction}

No-slip (or rough) billiards are a type of billiard dynamical system based upon a model in which linear and angular momentum of a hard spherical particle, moving in an n-dimensional Euclidean domain, may be exchanged on collisions at the boundary with total energy conserved, as indicated in the diagram of Figure \ref{fig:intro}.

\begin{figure}[htbp]
\begin{center}
\includegraphics[width=1.6 in]{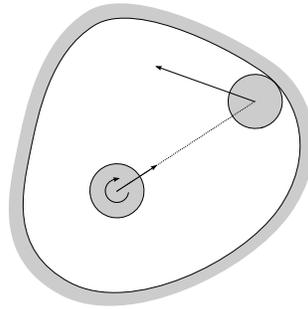}\ \ 
\caption{\small{In a  no-slip billiard system in dimension $2$
a form of non-dissipative friction at collisions  causes linear and rotational
velocities to be partially exchanged.}} 
\label{fig:intro}
\end{center}
\end{figure} 

Gutkin and Broomhead introduced no-slip collisions in two dimensions as a new model of a gas, in which ``the spheres interact with each other and with the container walls'' without slipping, in such a way that the impact ``conserves the total energy of the system, but mixes the tangential velocity components with the angular velocities of
colliding spheres.'' \cite{gutkin}  They demonstrate that in two
dimensions energy may be conserved only by the specular collisions of the traditional gas model, well-known from standard billiards, or the unique conservative alternative of no-slip collisions.

Very little (none to our knowledge)  seems to have been done regarding  the dynamics of no-slip billiards since  \cite{gutkin}. In particular,
the broader implications of the main observation of that paper, concerning a striking boundedness property that they obtain in the very special case of
 an infinite strip billiard table, were not pursued. The main goal of this paper is to reintroduce this   topic and begin a more systematic study of 
 the dynamics and ergodic theory of no-slip billiards. In addition to extending the boundedness result of \cite{gutkin}, we indicate by a combination of
 analytical and numerical  observations that this property should hold in much greater generality  than shown by Gutkin and Broomhead 
 and that it may serve as the basis of a very general non-ergodicity result for this class of billiard systems. We have intentionally  kept the  technical
 level of this paper fairly low. A more systematic development   will be taken up elsewhere. 
 
 \begin{figure}[htbp]
\begin{center}
\includegraphics[width=4.0in]{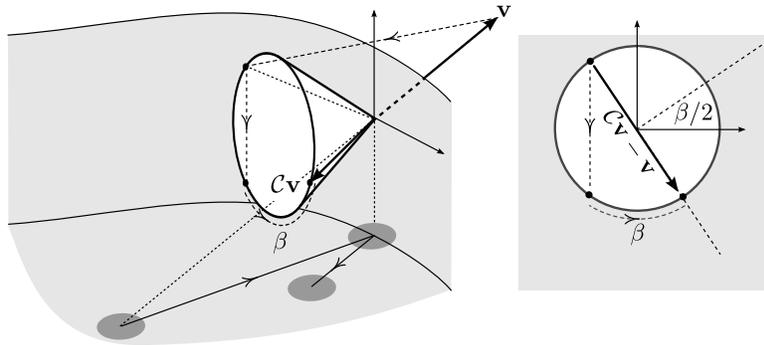}\ \ 
\caption{\small{Geometric description of the collision map
for no-slip billiards in dimension $2$. The vertical axis, perpendicular to the billiard table, represents the angle of rotation of the disk (linearly scaled so that total kinetic energy becomes the Euclidean norm).
The total (linear and angular) velocity of the billiard particle after collision,
$\mathcal{C}\mathbf{v}$, lies in the cone determined by the incoming velocity $\mathbf{v}$ as indicated in the figure. To obtain $\mathcal{C}\mathbf{v}$, first
reflect $\mathbf{v}$ back so as to point into the $3$-dimensional configuration space, 
change the sign of its rotational component, and finally rotate
the resulting vector by the angle $\beta$ such that $\cos\beta=1/3$ and
$\sin\beta=2\sqrt{2}/3$.}} 
\label{fig:conemap}
\end{center}
\end{figure} 
 
No-slip billiards in dimension $2$ arise naturally 
in   \cite{CFW}, in which we obtain an exhaustive description of collisions of 
rigid bodies in $\mathbb{R}^n$ satisfying  natural physical requirements. 
We summarize here the more relevant facts from that paper. 

\begin{definition}[Strict collision maps]
Let $M$  denote the configuration manifold of two rigid bodies in $\mathbb{R}^n$
having smooth boundaries. We endow $M$ with the Riemannian metric whose quadratic
form gives the system's  kinetic energy function  and assume that points $q\in \partial M$ represent configurations in which
  the bodies  have  a single contact point. 
At a boundary point $q$ of $M$ where $\partial M$ is differentiable, we define
a {\em collision map} as a linear map $\mathcal{C}:T_qM\rightarrow T_qM$
that sends vectors  pointing out of $M$ into  vectors pointing inward. 
We say that $\mathcal{C}$ is a {\em strict collision map} at $q$ if
\begin{enumerate}
\item Energy is conserved.  That is, $\mathcal{C}$ is an  orthogonal linear map;
\item Linear and angular momentum are conserved in the unconstrained motion. (This property, expressed in terms of invariance of a momentum map, may 
amount to no restriction at all  when one of the bodies, representing the billiard table, 
is assumed fixed in place and the subgroup of Euclidean symmetries of the whole system is trivial);
\item Time reversibility. This amounts to $\mathcal{C}$ being a linear involution;
\item Impulse forces at collision are applied only at the single point of contact.
(See \cite{CFW} for an elaboration and geometric interpretation of this property, which may be regarded as a generalized momentum
conservation law that is generally non-trivial and highly restrictive.)
\end{enumerate}
\end{definition}

The following result is shown in \cite{CFW}.

\begin{theorem}
At each boundary point of the configuration manifold of the system of two rigid
bodies, assuming the boundary is differentiable at that point, the set of
all strict collision maps can be expressed as the disjoint union of orthogonal Grassmannian manifolds $\text{\em Gr}(k, n-1)$, $k=0, \dots, n-1$,  of all $k$ dimensional planes
in $\mathbb{R}^{n-1}$. In particular, when $n=2$, the set of strict collision maps
is a two-point set consisting of the  specular reflection and the no-slip collision.  
\end{theorem}

The geometric description of  no-slip collisions is explained in Figure \ref{fig:conemap}. In addition to the standard reflection
it involves a rotation by    the special angle $\beta$, which
is the same for all table shapes. 

A billiard system is generally understood to consist of two rigid bodies, one of which is fixed in place  and called the {\em billiard table}. The
 focus is then on the motion of the second, referred to as the {\em billiard particle}. For the system to be fully specified it is necessary to
impose boundary conditions. From our perspective, this amounts to assigning   a collision map $\mathcal{C}_q$ to each boundary configuration
$q\in \partial M$ from among those in the moduli of collision maps described by the above theorem. We offer a variety of examples  in \cite{CFW}.

On a   general Newtonian mechanical system  without boundary  the so-called Liouville
volume measure,
obtained from the canonical (symplectic) form on the cotangent bundle
of the configuration manifold $M$,
 is invariant under the Hamiltonian flow.   It is also well-known that the Liouville measure remains invariant for systems with boundary when the collision maps correspond to  specular reflection, but it is not clear whether  this still holds for more
general boundary conditions. We show in \cite{CFW} that a sufficient condition for the invariance of the Liouville measure is that the field of collision maps $q\mapsto \mathcal{C}_q$ be parallel with respect to the Levi-Civita connection on $\partial M$
associated to the kinetic energy Riemannian metric. This condition is satisfied
for no-slip billiards on the plane.

\begin{theorem}
The   canonical billiard measure on the  boundary of the  phase space 
of a planar no-slip billiard is invariant under the billiard map.
\end{theorem}

Having established that no-slip billiard maps preserve the canonical billiard measure, we can venture into ergodic theory. 
We may ask, in particular, whether the known constructions giving rise to ergodic  billiards of the standard kind also apply to the no-slip kind.
The main result of \cite{gutkin} about boundedness of trajectories in infinite strip billiard tables, suitably generalized, suggests that ergodicity for no-slip billiards is very difficult to come by. 
A feature of no-slip billiards that may preclude ergodicity is an {\em axis of periodicity},  a certain type of period two trajectory
that will be seen below to be very ubiquitous. In contrast to periodic points that may exist in standard ergodic billiards, these axes appear to occur as the center of small invariant regions. The axes occur in a large class of billiards including polygons,
as will be seen.  
Furthermore, phase portraits suggest the dynamics persist under small smooth perturbations. Indications of the invariant regions precluding ergodicity appear in curvilinear polygons
in  numerical experiments shown here in Section \ref{ex}.
In this paper we provide evidence indicating that
\begin{itemize}
\item Gutkin and Broomhead's observation that trajectories in infinite strip billiards are bounded holds in much greater generality; 
\item This boundedness property is a strong obstacle preventing ergodicity.
\end{itemize}
This behavior demonstrates one marked difference between the dynamics of no-slip billiards and those of standard billiards. For example, generic polygons (having angles of irrational multiples of $\pi$) are ergodic \cite{KMS}, but no-slip billiards on polygons appear  never to be ergodic.

\vspace{0.1in}
\begin{figure}[htbp]
\begin{center}
\includegraphics[width=3 in]{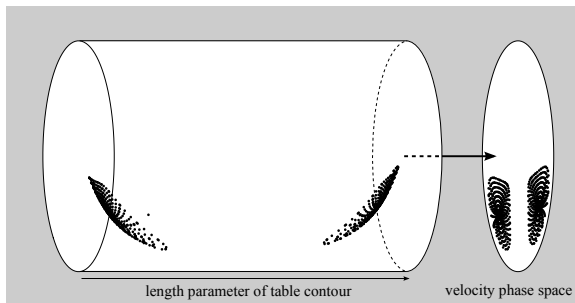}
\caption{\small{Reduced phase space for no-slip billiards in two dimensions is a solid torus, with the horizontal parameter being the position on the boundary. 
Disk cross sections give the possible velocities at that boundary point. Many of the diagrams shown below represent the projection
from the solid torus to a single disk by ignoring the length parameter along the boundary of the table. We call this disk the {\em velocity phase space}.}}
\label{fig:3d}
\end{center}
\end{figure}

Note that no-slip billiards in dimension $2$ correspond to configuration spaces $M$ of dimension $3$, whose points are parametrized 
 by the center of the disk-particle and
its  angle of rotation. The boundary of $M$ is then a two-dimensional, piecewise smooth manifold, and the phase space 
is a $4$-dimensional manifold whose points are pairs $(q,v)$ in which $q\in \partial M$ and $v$ may be taken to lie
in a hemisphere (whose radius is determined by the conserved kinetic energy) about the normal vector to $\partial M$ at $q$ pointing
into $M$. In this paper we identify this hemisphere  with the disk of same radius in the tangent space to $\partial M$ at
$q$ under the natural orthogonal projection.  Thus the phase space of our billiards, at least in the case of bounded (and connected) billiard tables, will typically be homeomorphic to
the Cartesian product of a $2$-torus and a disk.  A simplification results by noting that the angle of rotation (but certainly not the speed of rotation)
is typically immaterial. Formally this means that we may for many purposes consider the {\em reduced phase space}, defined as the quotient
of the $4$-dimensional phase space by the natural action of the rotation group $SO(2)$. The resulting $3$-dimensional  solid torus is indicated 
in the diagram of Figure \ref{fig:3d}.
 In many cases key feature may be gleaned from the projection to the velocity components, which we refer to as the velocity phase portrait. 
See Figure \ref{fig:hexa} for an example.

\begin{figure}[htbp]
\begin{center}
\includegraphics[width=3 in]{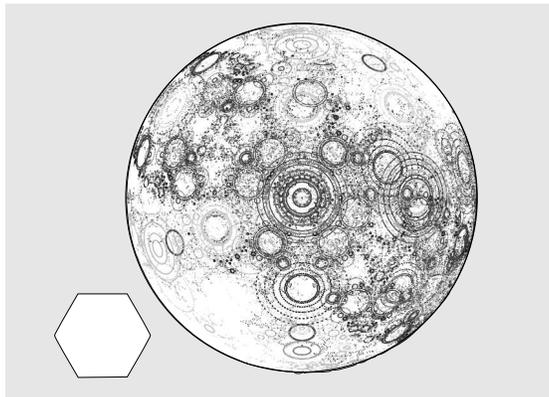}
\caption{\small{The velocity phase portrait is the projection of the $3$-dimensional reduced phase portrait as indicated in Figure \ref{fig:3d}. Here   orbits of the billiard map for the hexagon no-slip billiard (with table on the left) are shown projected to the velocity phase space.}}
\label{fig:hexa}
\end{center}
\end{figure}

In Section \ref{perbd} we consider several   examples of no-slip billiards in dimension two, beginning by revisiting the ``no-slip billiard on a strip'' from \cite{gutkin}. We give an alternate proof of boundedness which turns out to extend to dimension three  and show that there are no non-trivial periodic orbits. Turning to the open infinite  wedge, we show that all orbits not eventually escaping are bounded and that the wedge angle may be chosen for any $k \geq 1$ such that {\em all} bounded orbits are periodic with period $2k$.  The  kind of localization of orbits near what we have called {\em axis of periodicity} 
noted  here for wedge billiards
is observed for very general billiard shapes containing flat points, and perhaps many others. This suggests that a detailed  analysis of
wedge billiards and their smooth perturbations is fundamental for understanding the dynamics of general no-slip billiards.
Section \ref{ex} comprises examples of polygons, circles, and other simple shapes.

In standard billiards, two well-known techniques for generating chaotic dynamics are the defocusing and dispersing mechanisms. The stadium of Bunimovich \cite{bunim}, the most well-known example of chaotic focusing, consists of two half circles connected by flat segments. The corresponding no-slip billiard, however, fails to be ergodic, as any arbitrarily small flat strip will have bounded orbits for all trajectories in a neighborhood of positive measure near the vertical trajectory (Figure \ref{fig:stad-disp}, left). While apparently no such elementary argument can be used to show that the dispersing examples of standard billiards are not ergodic in the no-slip case, some insight is gleaned from the velocity phase portraits. Sinai billiards with boundaries which are entirely dispersing generate strong chaotic behavior in standard billiards \cite{sinai}. While numerical analysis of phase portraits of corresponding no-slip examples tends to show potentially chaotic behavior for most orbits, exceptional orbits forming closed curves suggest the possibility of non-ergodic behavior near periodic points. Figure \ref{fig:stad-disp} (right) shows a bounded component for one of the simplest dispersing cases, a disk on a flat torus.

\vspace{0.1in}
   \begin{figure}[htbp]
\begin{center}
\includegraphics[width=4.5in]{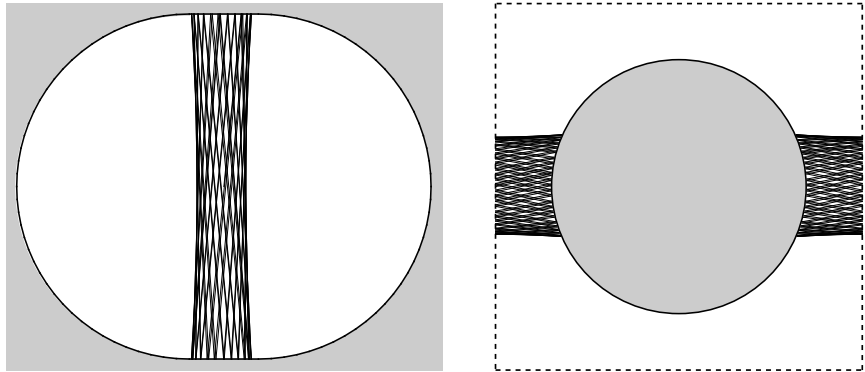}
\caption{\small{The two primary techniques of generating chaotic behavior in standard billiards, defocusing and dispersing, do not immediately translate to no-slip billiards. The Bunimovich stadium (left) is not ergodic, as the flat edges will always produce a positive measure region of phase space with bounded orbits. Similar bounded regions may appear in dispersing billiards with concave boundaries, like the Sinai type example of a disk on a torus (right).  }}
\label{fig:stad-disp}
\end{center}
\end{figure} 


\section{Periodic and bounded orbits}\label{perbd}

The basic properties obtained from the strip and wedge examples of this section will be important tools in general for understanding no-slip dynamics. If the unconstrained model for two rigid bodies in \cite{CFW} is adapted to one fixed body and a non-intersecting but otherwise unconstrained disk of uniform mass distribution, 
then the post-collision rotational and linear velocity $(v^+_0,v^+)$ is  the function of 
  the pre-collision velocities $(v_0^-, v^-)$ given by
\begin{equation}\label{update2}
\begin{aligned}
v_0^+&= -\frac13 v_0^- + \frac{2\sqrt{2}}{3} v\cdot (R_{\frac{\pi}{2}} \nu)\\
v^+&=\left[\frac{2\sqrt{2}}{3} v_0^- + \frac13 v^-\cdot (R_{\frac{\pi}{2}} \nu)\right] R_{\frac{\pi}{2}}\nu - (v^-\cdot \nu) \nu,
\end{aligned}
\end{equation}
where $R_{\theta}$ is the matrix of counterclockwise rotation by angle $\theta$ and $\nu$ is the inward normal vector at the point of contact.

For dimension two we may use coordinates ${x=(x_0,x_1,x_2)}^{\dagger}$, with planar position $(x_1,x_2)=(x,y)$ and normalized rotational position $x_0=\frac{R}{\sqrt{2}}\theta$, where $R$ is the radius of the disk. Note that the rotation adds a third dimension to the two spacial dimensions. Sometimes, if there is no loss of information, the given figures will be the planar projection. Also,  without any essential alteration of the theory we may use the orbit of the center of mass, with boundaries accordingly adjusted by distance $R$.  If the upward normal is in the direction $x_2$, then by \ref{update2} the pre- and post-collision velocities are related by $v^+=Tv^{-}$ where $v=(\dot{x_0}, \dot{x_1}, \dot{x_2})^{\dagger}$ and the transformation matrix $T \in O(3)$ is 
\begin{equation}
\label{eq:rt}
T={\left(\begin{array}{ccc}
-\frac{1}{3} &  \frac{2 \sqrt{2}}{3} & 0 \\[6pt]
 \frac{2\sqrt{2}}{3} & \frac{1}{3} & 0\\[6pt]
 0 & 0 & -1
\end{array}\right)}.
\end{equation}

  
\subsection{Parallel Boundaries}\label{ss:parallel}
As an application of no-slip collisions, Broomhead and Gutkin  \cite{gutkin} considered a disk between two parallel horizontal boundaries.

 \begin{figure}[htbp]
\begin{center}
\includegraphics[width=3.0in]{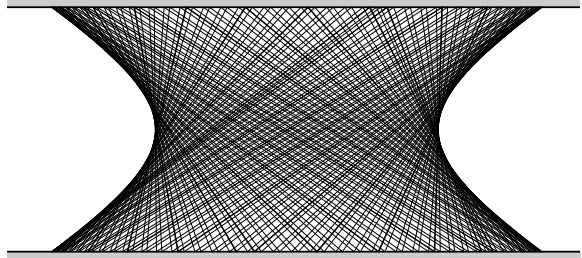}\ \ 
\caption{\small{No-slip collisions between planar parallel boundaries have maximal displacement $\sqrt{\frac{3}{2}\left( \frac{1}{(\dot{x}_2)^2} -1 \right)}$, corresponding
to the length of the base of the region in the above figure containing the trajectory. Here  $\dot{x}_2$ is the (constant) vertical velocity.}}
\label{fig:parallel}
\end{center}
\end{figure}

We normalize the separation to give unit time between collisions. Using complex coordinates for the phase space and finding a series for the horizontal displacement after $n$ collisions, it was  noted in \cite{gutkin} that the orbit is always bounded, except for the case of zero vertical velocity. The following strengthening of this main result
in \cite{gutkin} holds for unit velocity and fixed unit separation, letting $x_1$ be the parallel direction, so that the time between collisions will be $t=\frac{1}{|\dot{x}_2|}$, constant as $\dot{x}_2$ only changes sign at collisions.

   \begin{figure}[htbp]
\begin{center}
\includegraphics[width=3.5in]{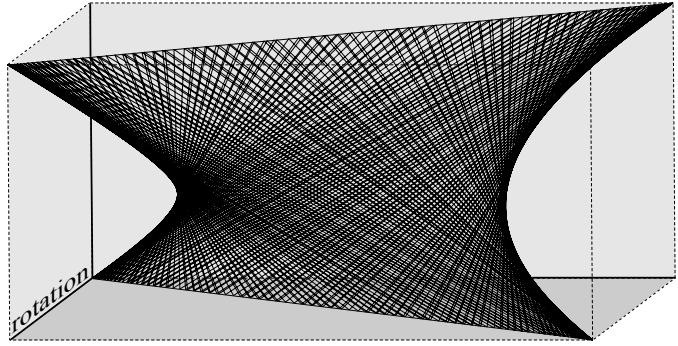}\ \
\caption{\small{One orbit segment of the no-slip strip viewed with rotational position as the third dimension. Figure \ref{fig:parallel} is the projection of this one to the plane of the billiard table.}}
\label{fig:3dview2dstrip}
\end{center}
\end{figure}

\begin{proposition}
\label{pr:horbd}
Orbits of a two dimensional system of no-slip collisions with parallel boundaries have horizontal displacement no more than $\sqrt{\frac{3}{2}\left( \frac{1}{(\dot{x}_2)^2} -1 \right)}$. 
\end{proposition}
\begin{proof}
Viewing the trajectories in three dimensions, we will show that for each boundary plane the set of collision points is contained in a line (Figure \ref{fig:3dview2dstrip}). 
Consider any two successive collision points $q^-$ and $q^+$ on the upper wall, choosing coordinates with the origin at the intermediate collision on the lower wall and letting $\dot{x}=(\dot{x}_0,\dot{x}_1,\dot{x}_2)$ be the velocity after the collision at $q^-$. Then  $t=-\frac{1}{\dot{x}_2}$ and \ref{eq:rt} gives $$q^-=\left(\frac{\dot{x}_0}{\dot{x}_2}, \frac{\dot{x}_1}{\dot{x}_2},1 \right)$$ and $$q^+=\left(\frac{\frac{1}{3}\dot{x}_0-\frac{2}{3}\sqrt{2} \dot{x}_1}{\dot{x}_2},  \frac{-\frac{1}{3}\dot{x}_1-\frac{2}{3}\sqrt{2} \dot{x}_0}{\dot{x}_2}, 1 \right).$$
The slope in the $x_1x_0$ (horizontal-rotational) plane is $$\frac{\Delta x_1}{\Delta x_0}=\frac{\frac{2}{3}\sqrt{2} \dot{x}_0+\frac{4}{3}\dot{x}_1}{\frac{2}{3} \dot{x}_0+\frac{2}{3}\sqrt{2}\dot{x}_1}=\sqrt{2},$$ independent of the initial conditions. Therefore, all of the collision points on the upper wall lie on a line of slope $\sqrt{2}$, and a similar argument on the lower wall shows that all collisions occur on a line of slope $-\sqrt{2}$.

As no velocity is exchanged between the horizontal-rotational component and the vertical component, the length of segments representing the projection of the orbits between successive collisions will be $\sqrt{ \frac{1}{(\dot{x}_2)^2} -1} $. Geometrically, the orbits are contained in an astroid (Figure \ref{fig:parastr}) with the given bound, the maximal horizontal displacement being achieved when the (projection of) the trajectories is perpendicular to the contact line. 
\end{proof}

\begin{remark}
The no-slip strip can never have more than three consecutive trajectory segments in the same horizontal direction: considering the trajectory segments in the upper, left, lower, and right quadrants of the $x_0x_1$ plane as delineated by the contact lines (as in Figure \ref{fig:parastr}), no two consecutive trajectories both lie in the upper (or lower) quadrants. At most two consecutive trajectories may lie in the left (or right) quadrants.
\end{remark}

In the case of three spacial dimensions, the argument of Proposition \ref{pr:horbd} may be generalized to two hyperplanes of the phase space in both of which the projected velocity is constant. The contact points occur on lines of slope $\sqrt{\frac{5}{2}}$ and similar bounds can be obtained in terms of the constant velocity in (or excluding) the hyperplanes, yielding the following result.

\begin{proposition} 
\label{pr:3dbdd}
Orbits of three dimensional no-slip collisions between parallel planes are bounded.
\end{proposition}

\begin{figure}
\begin{center}
\includegraphics[width=3.0 in]{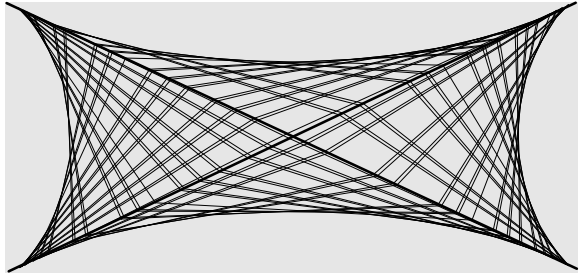}\ \ 
\caption{\small{For the no-slip strip, the projections of the orbits onto the horizontal-rotational plane are contained in an astroid. Each orbit must be a segment of fixed length in this plane from one contact line to the other.}}
\label{fig:parastr}
\end{center}
\end{figure}

Next we consider the question of periodic orbits in the strip case. If the velocity is entirely vertical there will be a period two orbit, where the disk's velocity is vertical with no rotation. This is a simple example of an axis of periodicity. 
It is useful to consider the composition of transformations on the velocity through a complete cycle of two collisions. In the no-slip strip case, $v^+=Sv^{-}$ where $S\in SO(3)$ is given by   
\begin{equation}
S=(FT)^2={\left(\begin{array}{ccc}
-\frac{7}{9} & -\frac{4}{9} \, \sqrt{2} & 0 \\[6pt]
\frac{4}{9} \, \sqrt{2} & -\frac{7}{9} & 0 \\[6pt]
0 & 0 & 1
\end{array}\right)}
\end{equation}
with $T$ as in equation \ref{eq:rt} and $F$ the appropriate frame adjustment.
The proof that there are no higher order periodic orbits will consider iterations of $S$ and use the following lemma of Niven \cite{N}.

\begin{lemma}
\label{lm:niv}
If $\theta$ is a rational multiple of $\pi$, the only possible rational values of $\cos(\theta)$ are $0,\pm \frac{1}{2},$ and $\pm 1$.
\end{lemma}

\begin{proposition}
The no-slip billiard system on the infinite strip  has no periodic orbits besides the trivial period-two orbit.
\label{pr:parnonper}
\end{proposition}

\begin{proof}
Notice that $v=Sv$ only if the velocity is vertical, giving the trivial period two orbit.
Suppose $\dot{x_0}\neq0$ or $\dot{x_1}\neq0$.  A necessary condition for higher order periodicity is then that $v=S^{n}v$ for some $n > 1$. 
$S \in SO(3)$ is a rotation matrix and will have finite order precisely when it is a rotation through an angle which is a rational multiple of $\pi$. But here the angle is $\alpha=\cos^{-1}(\frac{7}{9})$ is not a rational multiple of $\pi$ by Lemma \ref{lm:niv}.
\end{proof}

It follows from Proposition \ref{pr:parnonper} that the bounds of Proposition \ref{pr:horbd} are optimal. If $\dot{x}_2=1$ then the bound is $0$ which holds trivially for the simply periodic case. Otherwise, the orbit is not periodic and will come arbitrarily close to the geometric limit on the astroid by invariance of measure. Notice that for a certain initial velocity it is possible to have a ray-like orbit which achieves the limit on one side and then reverses direction. However, by the non-existence of higher order periodic orbits it cannot also achieve the limit on the other side.

The matrix $S$ and corresponding rotation angle $\alpha$ in the proof of Proposition \ref{pr:parnonper} will generalize. $S$ has an eigenvalue of $1$ corresponding to the eigenvector $(0,0,1)$, the axis of periodicity. Because the axis corresponds to a coordinate axis, $S$ is easily identified with an $S' \in SO(2)$, which suggests a natural connection to the \cite{gutkin} proof in terms of rotation in the complex plane. For the wedge example, though, the axis will have a component in the rotational direction; however, the axis and rotation angle will still be fundamental.

\subsection{Open wedges}\label{ss:wedge}
In this section we consider the case of no-slip billiards on an open wedge of angle $\theta$, letting $x_2$ be the direction of the outward bisector, $x_1$ the perpendicular spatial direction, and $x_0$ the rotational position. An informal numerical survey of wedge systems reveals that for most angles $\theta \in (0,\pi)$ the behavior is similar to the bounded, nonperiodic dynamics of the no-slip strip (Figure \ref{fig:wedgeexamples}). However, certain isolated $\theta$ have periodic orbits which are stable in the sense of persisting when the initial conditions are altered. These observations will be made rigorous in Proposition \ref{pr:bddwedge} and Corollary \ref{cor:anyper} respectively. But first we show the elementary but useful fact that all wedges have an axis of periodicity.

\begin{figure}[htbp]
 \centering
 \includegraphics[height=1.1in]{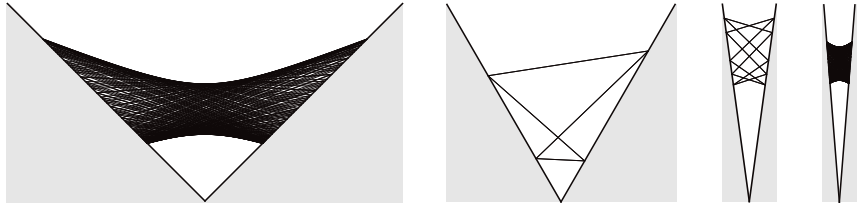}
 \caption{For most angles, no-slip open wedges give nonperiodic orbits (far left and far right), but $\theta=\frac{\pi}{3}$ (center left) and $\theta \approx .2709$ (center right) give period four and period ten orbits which persist when initial velocities are changed.}
 \label{fig:wedgeexamples}
 \end{figure}

\begin{proposition}
\label{pr:wedax}
An orbit of an angle $\theta$ wedge billiard with initial velocity $(\dot{x}_0,\dot{x}_1,\dot{x}_2)$ will have period two if and only if
\begin{equation}
\frac{\dot{x_0}}{\dot{x_1}}=-\sqrt{2}\sin\frac{\theta}{2}
\label{eq:spw}
\end{equation}
and $\dot{x_2}=0$.
\end{proposition}
\begin{proof}
Informally,  Figure \ref{fg:wsp} gives an intuition why such a direction must exist for the case $\theta=\frac{\pi}{2}$. 
\begin{figure}[b]
\centering
\includegraphics[height=.67 in]{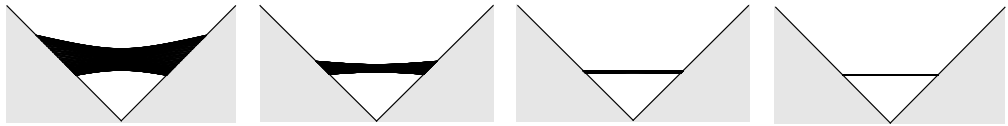}
   \caption{\small{The trajectories for the wedge angle $\frac{\pi}{2}$ and initial velocities (-1,1,$\dot{x}_2$), satisfying Equation \ref{eq:spw}, with $\dot{x}_2$=0.25, 0.1, 0.02, and 0. As $\dot{x}_2$ approaches zero the orbit approaches the axis of periodicity. } }
   \label{fg:wsp}
\end{figure}
Formally, consider $S_{\theta} \in SO(3)$  giving the transformation of the velocity after two collisions for a wedge of angle $\theta$. Adjusting the normal for the angle, we have $S_{\theta}=(R'_{\frac{\theta}{2}} T R'_{-\frac{\theta}{2}})^2 $, where $R'_{\theta}$ is the frame adjusted rotation.
Then $v=(-\sqrt{2}\sin\frac{\theta}{2}\dot{x_1},\dot{x_1},0)$ is an eigenvector of $S$ corresponding to the eigenvalue $1$. Since the outward component $\dot{x}_2$ is zero, the position as well as the velocity is unchanged.
\end{proof}

The no-slip strip case might be thought of as the limit as $\theta$ approaches zero. (Note, that the two spatial axes are reversed here: \ref{eq:spw} would imply $\dot{x}_0=0, \dot{x}_1 \neq 0$, giving the $x_2$ axis as the axis of periodicity.)

Turning to the question of higher periodicity for no-slip wedges, it is again useful to approach the question in terms of the velocity transformation $S_{\theta}$. In order to have an orbit of period $2n$ it is necessary for the velocity to return to the initial velocity after $n$ iterations, hence a necessary condition is $v=S_{\theta}^nv$. As $S_{\theta}$ simply rotates the velocity vector around the axis of periodicity, this condition will depend on the rotation angle $\alpha$ being a rational multiple of $\pi$. 
Notice that $w=(0,0,1)$ is orthogonal to the $v$ given in Proposition \ref{pr:wedax} giving the axis of periodicity. Then $\cos(\alpha)=v \cdot S_{\theta}v$. A direct calculation gives the following relation between $\theta$ and $\alpha$.

\begin{proposition}
\label{pr:wedgeangle}
The wedge angle $\theta$ and the rotational angle $\alpha$ of  $S_{\theta} \in SO(3)$, the corresponding transformation of the velocity after one cycle of two no-slip collisions, are related by

$$\frac{32}{9}\cos^4\left(\frac{\theta}{2}\right)-\frac{16}{3}\cos^2\left(\frac{\theta}{2}\right)+1=\cos(\alpha).$$
\end{proposition}

 \begin{figure}[t]

\begin{center}
\includegraphics[width=2.0 in]{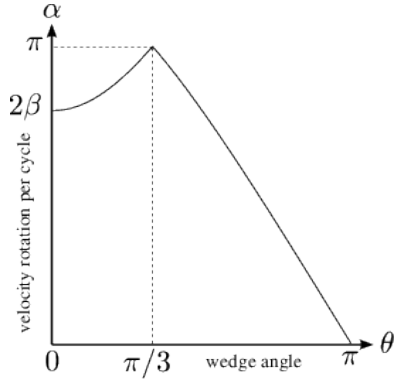}
\caption{\small{The velocity vector is rotated about the axis of periodicity by an angle $\alpha$, determined by the wedge angle $\theta$. }}
\label{fig:alphatheta}
\end{center}
\end{figure}

\begin{corollary}
\label{cor:anyper}
For any $n \geq 2$ there is an angle $\theta \in (0,\pi)$ such that all (nondegenerate) bounded orbits in an angle $\theta$ wedge have period $2n$.
\end{corollary}

\begin{proof}
By the proposition $\theta$ can be chosen so that the velocities are transformed by a rotation of $\alpha=\frac{\pi}{n}$ around the axis of periodicity in the half-sphere representation of the velocities, ensuring the necessary condition that the velocities are periodic. Preservation of measure implies that this condition is also sufficient to ensure the orbit is periodic. (See Lemma \ref{lem:repeat} below.) Degenerate cases may occur with period less than $2n$, but a small perturbation of initial conditions will yield an orbit of maximal period.
\end{proof}

\begin{corollary}
The set of wedge angles $\theta$ giving periodic orbits is dense in $(0,\pi)$.
\end{corollary}

\begin{figure}[htbp]
   \centering
\includegraphics[height=1.2 in]{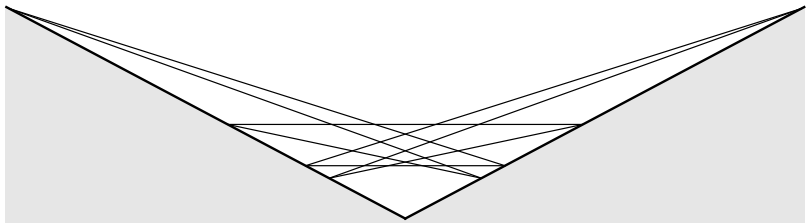}
   \caption{{\em For any $n \geq 2$ a wedge angle $\theta$ can be chosen such that all bounded orbits are periodic of period $2n$. For $n=4$, $\theta \approx 2.16598$ gives a (two-cycle) rotation angle $\alpha = \frac{\pi}{2}$ and period $8$ trajectories. }}
   \label{fig:perex}
\end{figure}

Figure \ref{fig:perex} shows the predicted period eight example. The degenerate cases occur where the spacial projection contains repeated segments; these  are resolved when the rotational dimension is incorporated. This may happen in any higher order periodic case with appropriate initial conditions.

Recall that the strip produced bounded orbits for all initial conditions except for the horizontal trajectory. With the wedge there are many trajectories that escape, but the analogous result below says that the velocity phase space may be similarly partitioned.
By the assumption of unit velocity, the set of velocities may be identified with the sphere $S^2$, excluding the two $\dot{x}_0$ poles representing the cases in which the there is only rotational motion. By time reversibility antipodal points correspond to the same trajectories and may be identified.
The existence of an axis of periodicity for all wedge angles $\theta$ implies the region of $S^2$ corresponding to bounded orbits is not empty. Numerically, one finds two antipodal non-escape regions with area a decreasing function of $\theta$, approaching $0$ as $\theta$ approaches $\frac{\pi}{2}$ and growing to $2 \pi$, an entire hemisphere, as $\theta$ approaches $0$. The latter limit is consistent with the fact that almost all initial conditions yield bounded orbits in the no-slip strip.

To specify these regions formally, let $E_{\theta}^0$ be the sector of the sphere corresponding to directions of direct escape from the $\theta$ wedge, along with the antipodal sector. Then for $S_{\theta} \in SO(3)$ we may define the escape region for $\theta$ as $$E_{\theta}=\bigcup_{n \in \mathbb{Z}} S_{\theta}^nE_{\theta}^0,$$ and its complement $E_{\theta}^c$ will be the region of non-escape velocities.

\begin{proposition}
\label{pr:bddwedge}
Every orbit with a non-escape velocity remains in the non-escape region.
\end{proposition}

\begin{proof}
First suppose $\theta$ is not periodic, that is, the corresponding angle of rotation for the transformation $S_{\theta}$ is not a rational multiple of $\pi$. Distance from the axis of periodicity is invariant under rotation by $S_{\theta}$, so the closest point in $E_{\theta}^0$ is rotated around a fixed spherical circle whose center is the axis point. Every point outside of it eventually is mapped to the escape sector, while no points inside are. Hence the escape and non-escape regions partition the sphere, except for a circular boundary. See Figure \ref{fig:escape}.

If $\theta$ is periodic of period $2n$, then the non-escape region $E_{\theta}^c$ is a curvilinear regular polygon, an n-gon if $n$ us odd and a 2n-gon in $n$ is even, which remains fixed with sides permuted under $S_{\theta}$.

\end{proof}

\begin{figure}[htbp]
\begin{center}
\includegraphics[width=5.0 in]{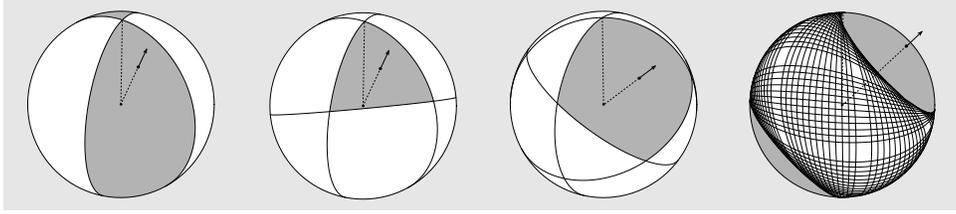}\ \ 
\caption{{\small The velocity space may be represented as a sphere for wedge systems. The velocity transformation rotates the escape sector (light area, left) around the axis of periodicity. The dark region for the right three spheres is the non-escape region, with polygonal boundary for periodic $\theta$ (left middle, right middle) and caustic circle for non-periodic $\theta$ (right).  }}
\label{fig:escape}
\end{center}
\end{figure} 

We can now prove the following lemma which was used above in establishing periodicity.

\begin{lemma}
\label{lem:repeat}
For a given orbit, if two collision points on the same edge of a no-slip wedge have identical velocities then their positions also correspond and the orbit is periodic. 
\end{lemma}

\begin{proof}
Suppose an orbit has a velocity vector repeated at distinct positions, and consider the trajectory of $k$ collisions between the two points oriented in the direction from the outer point to the inner point. The velocity pattern will then repeat the same cycle of $k$ trajectories as the orbit continues inward indefinitely. Consider the action of $k$ iterations of the billiard map on the subset of phase space consisting of the product of the non-escape velocities and a small rectangle in the rotational and spacial dimensions. The velocities are invariant while the rectangle contracts exponentially in the spacial dimensions and expands at most linearly in the rotational dimension, which violates preservation of measure.
\end{proof}


\section{Other examples of no-slip billiards}\label{ex}

\subsection{Circles}
In this section we give a characterization of the orbits of no-slip circular billiards. As noted in \cite{CFW}, such systems generally have double circular caustics.

\begin{figure}[htbp]
\begin{center}
\includegraphics[scale=.30]{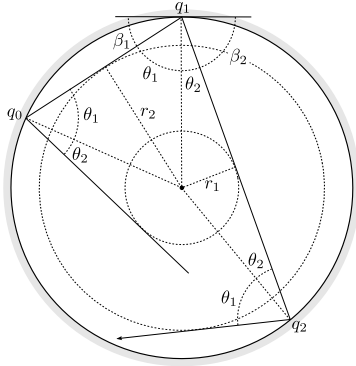}
\caption{\small{No-slip collisions result in alternating incident angles for circular billiards, with trajectories tangent to alternating circular caustics. }}
\label{fig:circlesetup}
\end{center}
\end{figure}

\begin{proposition}
For a   billiard system with circular table of radius $r$ and no-slip collisions, the projections of trajectories 
from the $3$-dimensional angle-position space to the disk in position plane have the property
that the vertex angle at each collision is a constant of motion. Moreover, for each projected trajectory $\gamma$, there
exists a pair of concentric circles of radius less than $r$
that  are touched  tangentially and alternately by 
 the sequence of
line segments of $\gamma$ at the middle point of these segments. 
\end{proposition}
\noindent
Implicit in the alternating caustics is the fact that the incoming angle relative to the tangent will alternate between angles $\beta_1$ and $\beta_2$ with successive collisions, or equivalently the internal angles between the trajectories and a radius will alternate between the complementary angles $\theta_1$ and $\theta_2$. (See Figure \ref{fig:circlesetup}.)

\vspace{0.1in}
   \begin{figure}[htbp]
\begin{center}
\includegraphics[width=3.5 in]{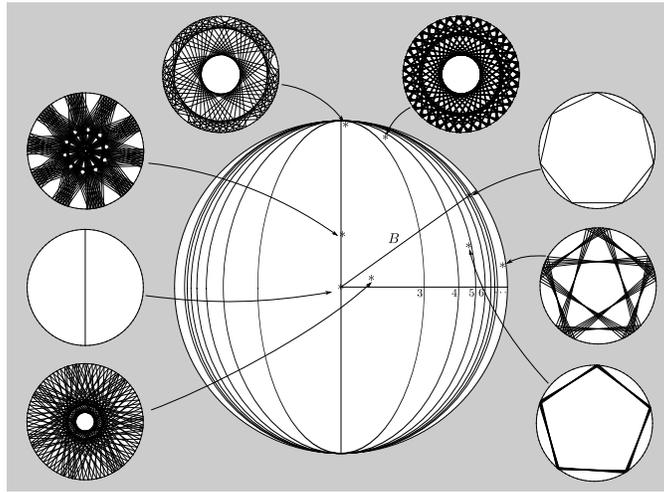}\ \ 
\caption{\small{Projection of velocity phase space in rotational and tangential direction, with examples. The two caustics converge along B yielding orbits identical to those of standard billiards.}}
\label{fig:circles}
\end{center}
\end{figure}

Consider three successive collisions, $q_0$, $q_1$, and $q_2$, and choose coordinates with $q_1$ at the origin, with rotational direction $x_0$, $x_1$ in the tangent direction, and $x_2$ in the radial direction. Let $v=(\dot{x}_0,\dot{x}_1,\dot{x}_2)$ be the velocity after the collision at $q_0$. All circular billiards up to a rotation or reflection may be considered by varying $v$ and changing $q_0$ and $q_2$ accordingly, leaving $q_1$ fixed, hence requiring $\dot{x}_2 \neq 0$. Then the orbit is determined by the choice of $v$, and in particular we can express the two caustic radii in terms of the components $v_i$. Immediately we have $r_1=\sin \theta_1$ and $\tan\theta_1=\frac{\dot{x}_1}{\dot{x}_2},$ and using the post-collision velocity $Tv$ we also have  $r_2=\sin \theta_2$ and $\tan\theta_2=\left(\frac{2\sqrt{2}}{3}\dot{x}_0+\frac{1}{3}\dot{x}_1\right)/(-\dot{x}_2).$
Hence
\begin{equation}
r_1=\frac{\dot{x}_1}{\sqrt{\dot{x}_1^2+\dot{x}_2^2}}
\end{equation}
and 
\begin{equation}
r_2=\frac{\frac{2\sqrt{2}}{3}\dot{x}_0+\frac{1}{3}\dot{x}_1}{\sqrt{(\frac{2\sqrt{2}}{3}\dot{x}_0+\frac{1}{3}\dot{x}_1)^2+\dot{x}_2^2}}
\end{equation}
Notice that in the above equations requiring the numerators to be equal guarantees the denominators are equal, thus we may ensure that $r_1=r_2$ (and $\theta_1=\theta_2$) by choosing $v$ with $\dot{x}_1=\sqrt{2}\dot{x}_0.$ If we further require $\dot{x}_2=\tan\left(\frac{\pi}{n}\right) \dot{x}_1$ the orbit will be a regular n-gon. 
These two requirements correspond respectively to segment $B$ and the numbered elliptic curves in  Figure \ref{fig:circles}, with regular n-gons at the intersections.
Additionally, the family of orbits with coinciding caustics $r_1=r_2$ are those in which there is no change in rotational velocity, a family corresponding precisely to the orbits of standard billiards. For standard billiards on a circle, or more generally on a smooth convex boundary, there is a countably infinite family of n-periodic billiards for any $n \geq 2$. (See \cite{serge}.) For no-slip billiards there is a continuum of initial conditions that give periodic orbits for a given $n$.

\begin{remark}
\label{rem:circerg}
For any boundary arc of angle greater than $\frac{\pi}{2}$, there is a positive measure region of phase for which the double caustics persist.
\end{remark}



\subsection{The equilateral triangle}

By  Proposition \ref{pr:wedgeangle}, the wedge angle $\theta = \frac{\pi}{3}$ yields period four (or degenerate period two) orbits. As a result, the equilateral triangle has unique dynamics. Immediately, any orbit which starts at one side and fails to contact both of the remaining sides by the third collision will be periodic. It turns out, in fact, that {\em all} orbits are periodic.

\begin{proposition}
\label{pr:equtri}
All nondegenerate orbits of the no-slip equilateral triangle billiard are periodic, with period two, three, four, or six.
\end{proposition}

\begin{figure}
 \begin{center}
\includegraphics[width=3.5 in]{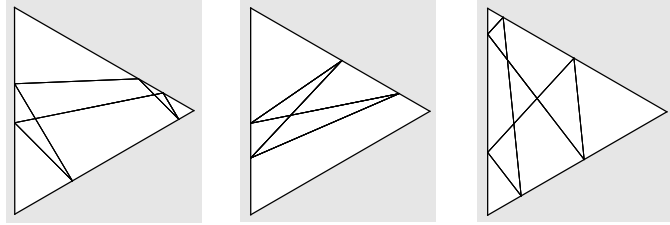}
 \label{fig:eq}
 \caption{{\small The no-slip equilateral triangle is periodic for all initial conditions. Orbits may be of the three types above or their degeneracies.}}
\end{center}
\end{figure}
\begin{proof}
First, we show that the velocies are periodic with the given possible values. Labeling the sides, consider the possible combinations keeping track only of which side the collision occurs on and not the exact location, looking at equivalence classes up to relabeling sides. Each possible orbit can be written as a sequence of ones, twos, and threes, without loss of generality beginning with $12$. 

Certain combinations are dynamically impossible: we can immediately rule out any sequence with repeated numbers, and $1213232$ can be ruled out because any $4-$cycle between two edges will repeat. Other combinations, like $12131$ can be ruled out by comparing the dynamics with the wedge dynamics when side $3$ is removed. The possible parameters when leaving side $1$ for the second time are such that it must, in the wedge, return to side two, parameters which also ensure the orbit would next hit side $2$ if instead it collides with wall $3$. Figure \ref{fig:eqproof} shows the combinatorial possibilities with dotted lines indicating equivalent forms and dynamically impossible combinations struck through.

Besides the four cycle $1212$, the only possibilities may be expressed $1213231$ and $1231231$. Fixing an orientation for the triangle, the transformation of the velocity vector after a collision, relative to the new frame, will either be $S_1=T R'_{\frac{\pi}{3}}$ or $S_2=T R'_{-\frac{\pi}{3}}.$ But the combined transformations for the two possibilities after the sixth collision are (in the given representation) $S_1 S_2^3 S_1^2=I$ or $S_1^6=I$, and the velocity is preserved.

\begin{figure}[htbp]
\label{fig:eqproof}
\begin{center}

\begin{forest}
[12
[121
[4-cycle, circle, draw]
[1213, name=spec
[\st{12131}]
[12132, name=spec2
[\st{121321}, name=spec3]
[121323
[6-cycle, circle, draw]
[\st{1213232}]
]
]
]
]
[123
[1231
[12312
[123121, name=dup3, draw]
[123123
[6-cycle, circle, draw]
[\st{1231232}]
]
]
[12313, name=dup2, draw ]
]
[1232, name=dup1,draw]
]
]
\draw[->,dotted] (dup1) to[out=north west,in=north] (spec);
\draw[->,dotted] (dup2) to[out=north west,in=north] (spec2);
\draw[->,dotted] (dup3) to[out=north west,in=north] (spec3);
\end{forest}

\caption{{\small Combinatorial possibilities for trajectories of the equilateral triangle. Strike through indicates combinations which as impossible with the no-slip collision map, while boxes indicate duplicates. }}
\end{center}
\end{figure}
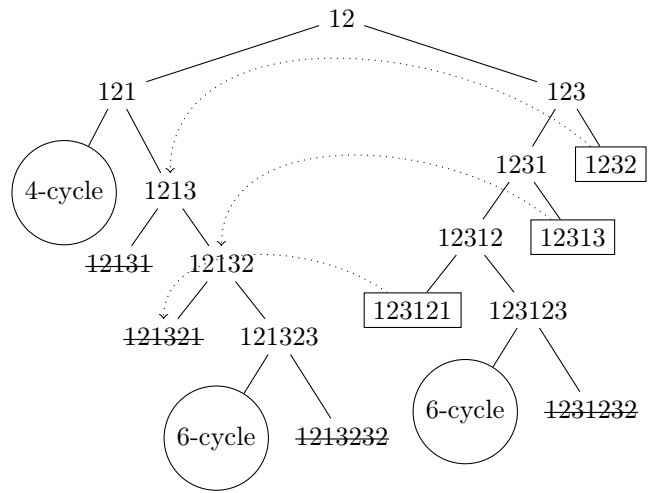

It remains to show that the orbits are actually periodic, that is, that the orbits return to the same boundary point and are not merely parallel trajectories.  Suppose that the orbit arrives back on the starting side but not at the same point, and consider subsequent returns after multiples of six collisions. The second set of trajectories will be parallel to the first all the way around and accordingly the sequence of walls will not change. Now consider the finite partition of the original wall by degenerate points where an orbit in the fixed direction from that point eventually hits a vertex. All return points must be in the same interval, always at the same distance apart and with the order preserved. But this implies that the entire interval is fixed, and the orbit must return to the starting point. 
\end{proof}


\subsection{Phase portraits}
\begin{figure}[htbp]
\begin{center}
\includegraphics[width=4.3 in]{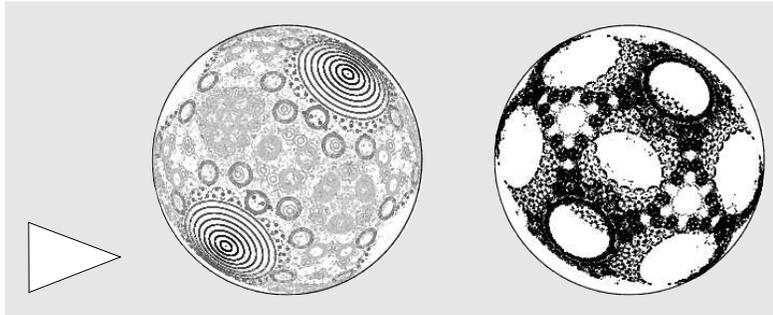}
\caption{\small{(Left) A sampling of phase space of the no-slip isosceles triangle. The axis of periodicity yields the two prominent pairs of concentric circles in the upper right and lower left, with many neighborhoods of higher order periodic points creating more complex symmetries. (Right) A single orbit near a high order periodic point. }}
\label{fig:isos}
\end{center}
\end{figure} 

The projection of the velocity phase portrait of an isosceles triangle no-slip billiard is given in Figure \ref{fig:isos}, where overlapping orbits are the result of the projection. This pattern is typical of polygons, with most orbits appearing as closed curves around periodic points, with the number of components matching the period of the central point. Numerical evidence suggests that for no-slip polygons there may be no positive measure ergodic components.

\begin{figure}[htbp]
\begin{center}
\includegraphics[height=1.8 in]{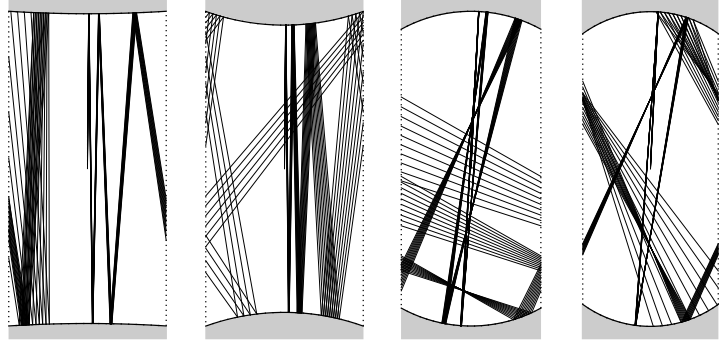}
\caption{\small{Straight edges are identified, so that collisions only occur at the curved boundaries. Left and middle left: examples of dispersion for standard and no-slip billiards. Middle right and right: examples of focusing. }}
\label{fig:focus}
\end{center}
\end{figure} 

Standard billiards with entirely concave boundaries are dispersing, with wave fronts expanding at every collision.
In contrast, the chaotic behavior of the Bunimovich stadium depends upon a defocusing mechanism: convex boundaries create a focusing front, but under circumstances in which the front passes through a focusing point and sees a net expansion before colliding at the next boundary. 
Looking at billiards with left and right edges identified and concave (or convex) upper and lower boundaries, both mechanisms appear to arise in the no-slip case as well. (See Figure \ref{fig:focus}.) To trigger a comparable level of defocusing after the same number of collisions for an equally spaced wave front, a greater curvature is required for the no-slip case. As the no-slip examples are projections, with independent expansion or contraction arising in the rotational dimension, the behavior may be more nuanced. 

\begin{figure}[htbp]
\begin{center}
\includegraphics[height=4 in]{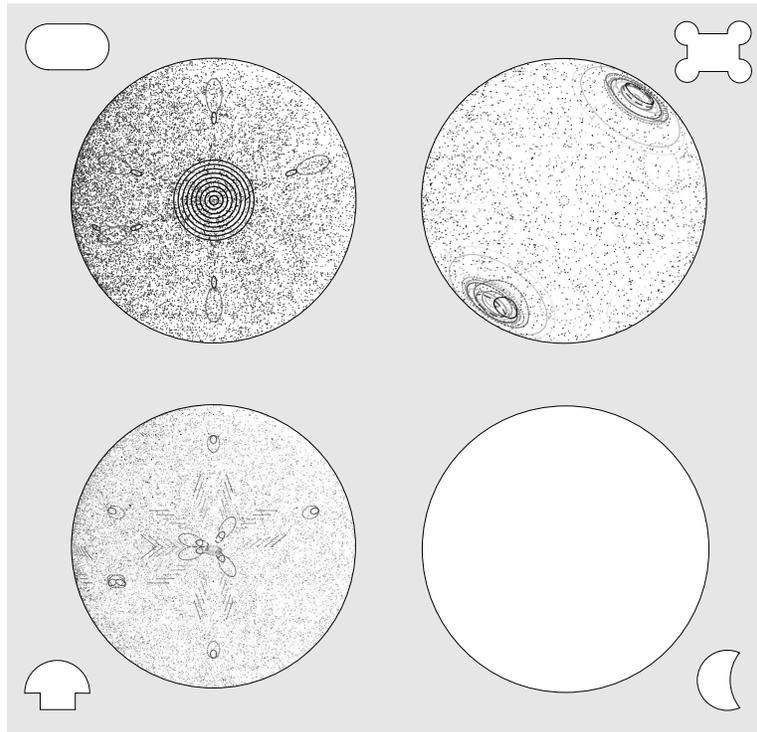}
\caption{\small{Velocity phase portraits of the stadium (upper left), the mushroom (lower left), a pocketed rectangle (upper right), and a moon billiard (lower right) exhibit closed orbits from bounded regions, but possibly large ergodic components as well.  }}
\label{fig:CombinedErg}
\end{center}
\end{figure} 

In spite of the apparent defocusing, the stadium has a bounded positive measure component phase space and is not ergodic. Similarly, other known standard billiard examples fail to be ergodic in the no-slip case: the `mushroom' billiard \cite{bunim01} has a similar bounded component in the stem, the `flower' billiard fails to be ergodic by Remark \ref{rem:circerg}, while `pocket' billiards are not ergodic for both of the above reasons. In these cases the no-slip phase portraits show a mixture of structured orbits mixed surrounded by potential chaotic seas (Figure \ref{fig:CombinedErg}).

The `moon' billiard, recently shown to exhibit ergodic behavior for some parameters \cite{CZ}, does not appear to be ergodic in the no-slip case. Along with the mushroom billiard, its phase portrait exhibits segments along which orbits linger before moving to apparent chaotic regions. This behavior is qualitatively similar to the known {\em marginally unstable periodic orbits} of standard billiards which both the moon \cite{CZ} and mushroom \cite{AFMK} \cite{DG} are known to exhibit.

Figure \ref{fig:hept} gives the phase portraits for three small but increasing $C^2$ perturbations of the no-slip regular pentagon, demonstrating the persistence of the dynamics. Numerical experiments using a range of polygons and a range of variation in curvature have not produced examples free from evidence of non-ergodic behavior.
 
\begin{figure}[htbp]
\begin{center}
\includegraphics[width=4.5 in]{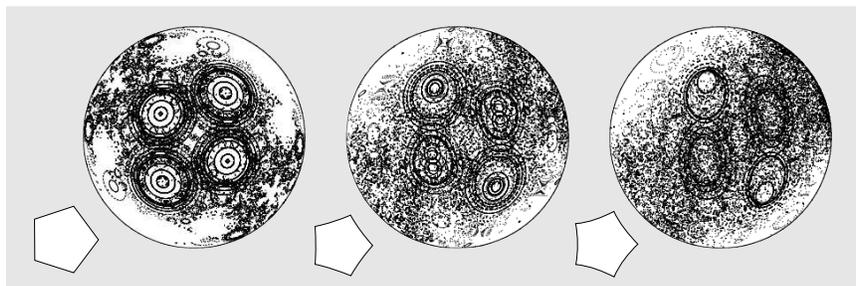}

\caption{\small{For small perturbations of the pentagon, the phase portraits suggest the non-ergodic dynamics persist.  }}
\label{fig:hept}
\end{center}
\end{figure}


\end{document}